 \documentclass[10pt]{amsart}
\usepackage{amsmath,amsthm,amssymb,amsfonts, eucal, amscd, mathbbol, mathrsfs}      
\usepackage{cite} 
\usepackage{setspace}
\usepackage[all,cmtip]{xy} 
\usepackage{graphicx}
\usepackage{subfig}  
\usepackage{fancyhdr}   
\usepackage{latexsym} 
\usepackage{epic}  
\usepackage{ifthen}
\usepackage[bookmarks,bookmarksnumbered, plainpages=false, pdfpagelabels]{hyperref}
\usepackage{rotating}

\usepackage{scalefnt}

\newtheorem{thm}{Theorem}
\newtheorem{cor}[thm]{Corollary} 
\newtheorem{lem}[thm]{Lemma} 

\theoremstyle{definition} 
\newtheorem{defn}[thm]{Definition}
\theoremstyle{definition} 

\theoremstyle{definition} 

\theoremstyle{definition}
\newtheorem{remarks}[thm]{Remarks}
\theoremstyle{definition}

\theoremstyle{definition} 

\theoremstyle{definition}

\numberwithin{thm}{subsection}

\newtheorem*{thm*}{Theorem}
\newtheorem*{cor*}{Corollary}

\newcommand{\R}{\ensuremath{\mathbb{R}}}

\newcommand{\<}{\langle} 
\renewcommand{\>}{\rangle} 
 
\def\p{\partial} 
 
\def\i{\infty}
\def\a{\alpha}
\def\e{\epsilon}

\def\det{\it{det}} 
\def\supp{\it{supp}}

\def\b{\beta} 
\def\g{\gamma}

\def\L{\Lambda} 
\def\o{\omega} 
\def\O{\Omega} 
\def\s{\sigma} 
\def\t{\tau} 
\def\D{\Delta}

\def\cal{\mathcal}

\def\rotateminus{\reflectbox{\rotatebox[origin=c]{155}{\hspace{.6pt}-}}}

\def\Xint#1{\mathchoice
{\XXint\displaystyle\textstyle{#1}}%
{\XXint\textstyle\scriptstyle{#1}}%
{\XXint\scriptstyle\scriptscriptstyle{#1}}%
{\XXint\scriptscriptstyle\scriptscriptstyle{#1}}%
\!\int}
\def\XXint#1#2#3{{\setbox0=\hbox{$#1{#2#3}{\int}$}
\vcenter{\hbox{$#2#3$}}\kern-.5\wd0}}

\def\cint{\Xint    \rotateminus }

\begin{document} 
	\title[Geometric Poincar\'e Lemma]{Geometric Poincar\'e Lemma }  
\author[J. Harrison]{J. Harrison  \\Department of Mathematics \\University of California, Berkeley} 

\thanks{The author was partially supported by the Miller Institute for Basic Research in Science and the Foundational Questions in Physics Institute}
 
\maketitle                     
        
 \section{Introduction}  
\label{sec:introduction}

  						 The Poincar\'e Lemma for the de Rham complex \(\O(U) = \oplus_{k=0}^n \O^k(U)  \) of smooth differential forms defined on contractible open subsets \( U \) of \( \R^n \) is a foundational result with applications from topology to physics: \emph{Every closed  differential  form \( \o \in \O_k(U) \) is exact}.  That is,  if \( d \o = 0 \)  where \( \o \in \O_k(U) \), then there exists  $\eta \in \O_{k-1}(U) $   such that $\o = d\eta.$    
						Such forms $\eta$ are found through a homotopy operator $A$  acting on forms where $\eta = A \o$.   That such an operator $A$  exists is a remarkable and powerful feature of the de Rham complex.  
						
						A  homotopy operator \( K \) on the chain complex of  polyhedral $k$-chains  is easy to construct using the classical cone construction from topology, and works quite well for the category of polyhedral chains.  One way to generalize Poincar\'e's Lemma is to introduce a topology on each vector space of polyhedral $k$-chains with the hope of extending \( K \) to  a continuous homotopy operator on the chain complex of the completed spaces.    For a coherent theory with broad application to domains of integration going far beyond polyhedral chains, we need the dual space to be an identifiable space of differential $k$-forms, and the dual operator \( A \o := \o K \), which is necessarily continuous since \( K \) is continuous,  to be a computable homotopy operator.    
 						
						Whitney's  Banach space of ``sharp chains''  \cite{whitney} has no continuous boundary operator, and thus there is  no  meaningful cone operator in the sharp space.  The cone operator \( K \) on polyhedral chains is not generally continuous in the Banach space of ``flat chains'', although it is continuous in the  subspace of flat chains with finite mass, and this has been useful as a tool for solving a special case of Plateau's problem for integral currents    (\cite{federerfleming}, see also \cite{fleming}).      Our efforts to solve the Plateau problem in full generality \cite{plateau10} led to the results in this paper.                                  

						 Our main result is a  geometric Poincar\'e Lemma for a certain differential chain complex (to be defined below) of topological vector spaces \(  \hat{\cal B}_k(U_1) \) of ``differential  $k$-chains''    in an open set \( U_1 \)  that is contractible in an open set \( U_2  \)  with \( U_1 \subset U_2 \subset \R^n \).    That is, there exists \( F:[0,1] \times U_1 \to U_2 \)   with \( F(0,p) = p \) and \( F(1;p) = c \) for some constant \( c \in U_2 \). 
						\vspace{-0.2in}\begin{center}
						 \emph{Every cycle \( J \in \hat{\cal B}_k(U_1) \)   is a boundary in \( \hat{\cal B}_k(U_2) \) for all \( 1 \le k \le n-1 \).} 
						\end{center} \vspace{-0.2in} That is,  if \( J \in \hat{\cal B}_k(U_1) \) satisfies \( \p J = 0 \), there exists \( K \in \hat{\cal B}_{k+1}(U_2) \) with \( \p K = J \).  (See Theorem \ref{thm:poincarelemma} in \S\ref{PL}.)       
						
						As a corollary, we obtain the following dual result.  Let \( {\cal B}_k(U) \) be the Frech\'et space of differential forms defined on \( U \), each with uniforms bounds on each of its directional derivatives.  Then \( {\cal B}_k(U) \) is the topological dual of \(  \hat{\cal B}_k(U) \) (see \cite{OC} Theorem 2.12.8).
					\vspace{-0.2in}\begin{center} 	
			 \emph{Every closed form \( \o \in {\cal B}_k(U_2) \)  is exact in  \( {\cal B}_{k-1}(U_1) \), for all		   \( 1 \le k \le n-1 \).}  \end{center}
			 \vspace{-0.2in} That is,  if $\omega \in {\cal B}_k(U_2)$ satisfies \( d \o = 0 \), there exists \( \eta \in {\cal B}_{k-1}(U_1) \)  with $\omega|_{U_1} = d \eta$.

	 There are three primary reasons for calling our result a ``geometric Poincar\'e Lemma''.  First of all, the topology of \( \hat{\cal B}_k(U) \) is defined constructively, and is not simply the abstract dual of the space of currents \( {\cal B}_k(U)' \). Similarly, the cone operator \( K \) is defined geometrically. Finally, Dirac chains are dense in \( \hat{\cal B}_k(U) \), yielding a  discrete and computable version of the geometric Poincar\'e's lemma and its cone operator \( K \).  The operator \( K \) restricts to the classical cone operator on polyhedral chains, and thus our theorem is a generalization of the classical version of Poincar\'e's Lemma for chains.      

 						Applications presented in this paper include a broad generalization of the Intermediate Value Theorem to arbitrary dimension and codimension (see Theorem \ref{ivt} and Figure \ref{fig:funky}), a new approach to homology theory  with potential extensions to non-manifolds \cite{harrisonhirsch}.   H. Pugh uses results of this paper to find generalizations of the Cauchy integral theorems \cite{cauchy}, and, in a sequel  \cite{plateau10}, the author uses the geometric Poincar\'e Lemma  to provide a general solution to Plateau's problem. 

The author is indebted to M.W. Hirsch and H. Pugh for helpful comments.    

    \renewcommand{\labelenumi}{(\alph{enumi})}
    \section{Preliminaries} 
\label{sub:preliminaries}    	  
   This work relies on new methods of calculus presented in \cite{OC}.  In this preliminary section we recount the definition of the topological vector space \( \hat{\cal B}(U) \) and its operators  from \cite{OC}  which we use below.        
 
 \subsection{The Mackey topology on Dirac chains}  
 \label{sub:subsection_name}
 
Let \( {\cal A}_k(U) \) be the vector space of \emph{Dirac $k$-chains}, i.e., formal sums \( \sum (p_i; \a_i) \) where \( p_i \in U \) and \( \a_i \in \L_k(\R^n) \).  We call \( (p;\a) \) a \emph{simple} $k$-\emph{element} if \( \a \) is a simple $k$-vector.  Otherwise, \( (p;\a) \) is called a $k$-\emph{element}.    Since \( ({\cal A}_k(U), {\cal B}_k(U)) \) is a dual pair, the Mackey topology \( \t_k \) is uniquely determined on  \( {\cal A}_k(U) \).   (The Mackey topology \( \t_k \) is the finest topology \( \mu_k \) on \( {\cal A}_k(U) \) such that  \( ({\cal A}_k(U), \mu_k)' = {\cal B}_k(U) \).) This process may be mimicked for any subspace of  forms, e.g., \( {\cal D}_k(U) \), the space of smooth functions with compact support in \( U \),  but \( {\cal B}_k(U) \) is especially nice to work with because of its useful algebra of continuous operators, and algebraic features of its topology.  The complex \( \hat{\cal B}_k(U)  \), which can also be defined for open subsets of Riemannian manifolds, is intrinsic in the following sense recently announced in    \cite{topological}:  
		  
		\begin{thm*}\label{thm:topo}
		  	The topology \( \t_k(U) \) is the finest topology in the collection \( \{\mu\} \) of  locally convex Hausdorff topologies  on Dirac chains ${\cal A}_k(U)$    satisfying three axioms:
			\begin{enumerate}
				\item The topological vector space $({\cal A}_k(U),\mu)$ is bornological;
				\item $K^0= \{(p;\a)\in {\cal A}_k(U) : \|\a\|=1\}$ is bounded where $\|\a\|$ is the mass norm of $\a$;
				\item  The linear map $P_v: ({\cal A}_k(U),\mu)\rightarrow\overline{({\cal A}_k(U),\mu)}$ determined by $P_v(p;\a):=\lim_{t\rightarrow 0}(p+v;\a/t)-(p;\a/t)$ is well-defined and uniformly bounded for all unit vectors $v\in \R^n$.
				\end{enumerate}
		\end{thm*}    

 The first two properties are essential to analysis in this space, the third suffices for the  Lie derivative of differential forms to be a bounded operator.    It is shown in \cite{topological} that the Mackey topology \( \t_k \) coincides with the constructive definition of \( \t_k \) of \cite{OC}, which we next recall. 
\subsection{Constructive definition of the \( B^r \) norms}  
\label{sub:constructive_defn}

 \subsubsection{Mass norm} \label{ssub:differential_chains}
     An inner product \( \<\cdot,\cdot\> \) on \( \R^n \) determines the mass norm on \( \L_k(U) \) as
		follows: Let \( \<u_1 \wedge \cdots \wedge u_k, v_1 \wedge \cdots \wedge v_k \> = \det(\<u_i,v_j\>) \).  The	\emph{mass} of a simple $k$-vector \( \a = v_1 \wedge \cdots \wedge v_k \) is defined by \(
		\|\a\| := \sqrt{\<\a,\a\>} \). The mass of a $k$-vector \( \a \) is \( \|\a\| := \inf\left\{\sum_{j=1}^{N} \|(\a_i)\| : \a_i \mbox{ are simple, } \a =
		\sum_{j=1}^N \a_i \right\}. \) Define the \emph{mass} of a $k$-element \( (p;\a) \) by  \(\|(p;\a)\|_{B^0} := \|\a\|  \). Mass is a norm on the subspace of Dirac $k$-chains
		supported in \( p \), since that subspace is isomorphic to the exterior algebra \(
		\L(\R^n) = \oplus_{k=0}^n \L_k(\R^n) \) for which mass is a norm (see \cite{federer}, p 38-39). The \emph{mass} of a Dirac $k$-chain \( A =  \sum_{i=1}^{m} (p_i; \a_i) \in {\cal A}_k(U) \)   is given by  
	   \[  \|A\|_{B^0} := \sum_{i=1}^{m} \|(p_i; \a_i)\|_{B^0}. \] If a different inner product is chosen, the resulting masses of Dirac
		chains   are topologically equivalent.    It is straightforward to show that \( \|\cdot\|_{B^0} \) is a norm on \( {\cal A}_k(U) \).  

		\subsubsection{Difference chains and the \( B^r \)
	  norm on chains} \label{sub:difference_chains}

		 Given a $k$-element \( (p;\a) \) with \( p \in U \) and \( u \in \R^n \), let \( T_u(p;\a) := (p+u;\a) \)
		be translation through \( u \), and \( \D_u(p;\a) := (T_u -I)(p; \a) \). Let \( S^j =
		S^j(\R^n) \) be the \( j \)-th \emph{symmetric power} of the symmetric algebra \( S(\R^n)
		\). Denote the symmetric product in the symmetric algebra \( S(\R^n) \) by \( \circ \). Let \( \s = \s^j = u_j \circ \dots
		\circ u_1 \in S^j \) with \( u_i \in \R^n, i = 1, \dots, j \). Recursively define $\D_{u
		\circ \s^{j}}(p; \a) := (T_u - I)(\D_{\s^j}(p; \a))$. Let \( \|\s\| := \|u_1\| \cdots
		\|u_j\| \) and \( |\D_{\s^j}(p; \a)|_{B^j} := \|\s\| \|\a\|. \)  Let \( \D_{\s^0} (p;\a) := (p;\a) \), to keep the notation consistent.   We say \( \D_{\s^j} (p;\a) \) is \emph{inside} \( U \) if the convex hull of \( \supp( \D_{\s^j} (p;\a)) \) is a subset of \( U \).

\begin{defn} \label{def:norms} For \( A \in {\cal A}_k(U) \) and \( r\ge 0 \), define the seminorm \[   \|A\|_{B^{r,U}} := \inf \left \{ \sum_{j=0}^r \sum_{i=1}^{m_j} \|\s_{j_i}\|\|\a_{j_i}\|: A = \sum_{j=0}^r \sum_{i=1}^{m_j} \D_{\s_{j_i}^j}(p_{j_i};\a_{j_i}) \mbox{ where } \D_{\s_{j_i}^j}(p_{j_i};\a_{j_i})  \mbox{ is inside } U \right\}. \] 
\end{defn}   For simplicity, we often write \( \|A\|_{B^r} =  \|A\|_{B^{r,U}} \) if \( U \) is  understood.   It is easy to see that the \( B^r \) norms on Dirac chains are decreasing as \( r   \) increases.

    It is shown in \cite{OC} (Theorem 2.6.1)  that \( \|\cdot\|_{B^r} \) is a norm on the free space of Dirac $k$-chains \( {\cal A}_k(U) \) called the \emph{\( B^r \) norm}.  
 		Let \( \hat{\cal B}_k^r(U) \) be the Banach space obtained upon completion of \( {\cal A}_k(U) \) with the \( B^r \) norm.  Elements of \( \hat{\cal B}_k^r(U), 0 \le r \le \i \),  are called \emph{differential\footnote{The adjective ``differential'' here does not indicate that the chains are in any way smooth, only that they are closed under the topological dual to the Lie derivative operator of differential forms.}   $k$-chains of class \( B^r \) in \( U \)}.   Let \(\hat{\cal B}_k(U)  =  \hat{\cal B}_k^\i(U) := \varinjlim \hat{\cal B}_k^r(U) \) be the inductive limit as \( r \to \i \), endowed with the inductive limit topology.     It is shown in \cite{topological} that \( \overline{\hat{\cal B}_k(U)} \cong \overline{({\cal A}_k(U), \t_k)} \).  Therefore, the inductive limit topology coincides with the Mackey topology \( \t_k \). 

  Let \( {\cal B}_k^0(U) \) be the Banach space of bounded measurable $k$-forms, \( {\cal B}_k^1(U) \) the Banach space of Lipschitz $k$-forms, and for  each  \( r > 1 \),   \( {\cal B}_k^r(U) \) be the Banach space of differential $k$-forms, each with bounds on each of the \( s \)-th order directional derivatives for \( 0 \le s \le r-1 \) and the \( r \)-th derivatives satisfy a Lipschitz condition.    Denote the resulting norm by \( \|\cdot\|_{B^r} \).  Elements of \( {\cal B}_k^r(U) \) are called \emph{differential   $k$-forms of class \( B^r \) in \( U \)}.     The natural inclusions \(  \hat{\cal B}_k^r(U_1) \hookrightarrow  \hat{\cal B}_k^r(U_2) \) are continuous for all open \( U_1 \subset U_2 \subset \R^n \). 
  Let \( {\cal B}_k(U) = {\cal B}_k^\i(U) = \varprojlim {\cal B}_k^r(U)  \) be the projective limit as \( r \to \i \), endowed with the Frech\'et topology.  If \( U_1\subset U_2 \),   the natural restrictions \( {\cal B}_k^r(U_2) \to {\cal B}_k^r(U_1)  \) are continuous, but they are not surjective unless \( U_1 \) is a regular open subset, that is, \( U_1 \) equals the interior of its closure.     It is shown in \cite{OC} Theorem 2.8.2 that \(  (\hat{\cal B}_k^r(U))' \cong {\cal B}_k^r(U)  \), and the integral pairing \( \cint: \hat{\cal B}_k^r(U) \otimes {\cal B}_k^r(U) \to \R \) where \( J \otimes \o \mapsto \o(J) \) is nondegenerate.  Let  \( \cint_J \o := \o(J) \) for all \( J \in \hat{\cal B}_k^r(U) \) and \( \o \in  {\cal B}_k^r(U) \).

	If \( J \) is nonzero then its support \( \supp(J) \) is a uniquely determined nonempty set (see \cite{OC} Theorem 6.1.3).           The \emph{support}  of a nonzero $k$-chain \( J \in \hat{\cal B}_k(U) \) is the smallest closed subset \( E \subset U \)   such that  \( \cint_J \o = 0 \) for all smooth \( \o \) with compact support disjoint from \( E \)  (see \cite{OC} Theorem 6.2.2).

 \subsection{Operators}  
\label{sub:operators}

 Let \( {\cal B}(U) := \oplus_{k=0}^n {\cal B}_k(U) \). Since     \(  \hat{\cal B}_k(U) \) is a subspace of currents \( {\cal{B}}_k(U)' \)  which is proper if \( U = \R^n \) (see \cite{topological}), a natural question arises:  If \( T: {\cal B}(U_1) \to {\cal B}(U_2) \) is a continuous operator on forms, we know its dual \( S(\o) := \o(T) \) is continuous on currents   \(S: {\cal B}(U_2)' \to {\cal B}(U_1)' \).  Is \( S: \hat{\cal B}(U_2) \to \hat{\cal B}(U_1) \) closed?     The constructive definition of the Mackey topology \( \t_k \) is often useful for answering such questions as seen in the following Lemma.  

\begin{lem}\label{lem:test}
	 Fix \( 0 \le s \le r \)  and \( 0 \le k \le n \). If \( T: {\cal A}(U_1) \to {\cal A}(U_2) \) is a graded operator  
	satisfying \[ \|T(\D_{\s^j} (p;\a))\|_{B^r} \le C \|\s\|\a\| \] for some constant \( C > 0 \)  and all $j$-difference $k$-chains \(  \D_{\s^j} (p;\a) \)  inside \( U_1 \) with \( \s^j \in S^j(\R^n) \),     \( 0 \le j \le s  \), \( \a \in \L_k(\R^n) \), then \( \|T(A)\|_{B^r} \le C\|A\|_{B^s} \) for all \( A \in
	{\cal A}_k(U_1) \).
\end{lem}
      See \cite{OC}, Theorem 2.7.3.

   Define \( E_v (p;\a) = (p;v \wedge \a) \).  Suppose \( \b = v_1 \wedge \cdots \wedge v_s \) is simple.  Define
	  \( E_\b = E_{v_s} \circ \cdots \circ E_{v_1} \)    and \( i_\b = i_{v_1} \circ \cdots \circ i_{v_s} \).    
	
\begin{thm}\label{thm:A}
	 Fix \( \b \in \L_s(\R^n) \). Then \( E_\b: \hat{\cal B}_k^r(U) \to \hat{\cal B}_{k+s}^r(U)   \) and \( i_\b: {\cal B}_k^{r+s}(U) \to {\cal B}_k^r(U)  \) are continuous
		graded operators with \( \cint_{E_\b J} \o = \cint_J i_\b \o \). Furthermore, \(  \|E_\b(J)\|_{B^r} \le  \|\b\|\|J\|_{B^r} \).
\end{thm}  
   See \cite{OC} Theorem 3.1.3 for a proof.
		
						    Boundary \( \p = \p_k \) coincides with the classical boundary operator \( \p_k \) on polyhedral $k$-chains, which are dense in \( \hat{\cal B}_k(U) \), although a direct definition using ``higher order Dirac chains'' is provided in \cite{OC} \S 4.   
						
\begin{thm}\label{thm:B} [General Stokes' Theorem]
The boundary operator on differential chains \( \p:\hat{\cal B}_{k}^r  \to  \hat{\cal B}_{k-1}^{r+1} \) is continuous, \( \p \circ \p = 0 \), and       										  \( \|\p J\|_{B^{r+1}} \le kn \|J\|_{B^r} \)   for all \( J \in \hat{\cal B}_k^r \).   Furthermore, exterior derivative  \( d: {\cal B}_{k-1}^{r+1}(U) \to {\cal B}_k^r(U) \) is continuous. If \( \o \in {\cal B}_{k-1}^r \) is a differential form and \( J \in \hat{\cal B}_{k}^{r-1} \) is a differential chain, then   \[ \cint_{\p J} \o = \cint_J d \o. \]
\end{thm} 
   	  See \cite{OC} (Theorems 4.1.3 and 4.1.6).
										  
 											 We say a differential $k$-chain \( \g \in \hat{\cal B}_k(U)  \) is a \emph{differential	$k$-cycle} if $\p \g = 0.$

		Differential $0$-forms \( f  \) determine functions \( f:U \to \R \) by \( f(p) := f(p;1) \).  We say that the function \( f \) is \emph{of class \( B^r \)} if the $0$-form \( f \in {\cal B}_0^r \).  In (see \cite{OC} \S 2.9 ) we prove that \( f \) is of class \( B^r \) if and only if  \( f \) is of class \( C^{r-1+Lip} \), i.e., \( f \) is  \( (r-1) \)-times continuously differentiable, and its  \( (r-1) \)-st directional derivatives are Lipschitz. 	 For each \( f:U \to \R   \) of class \( B^r \) define a map \( m_f:{\cal A}_k(U) \to {\cal A}_k(U) \) by \( m_f(p;\a) := (p; f(p)\a) \) for simple $k$-elements \( (p;\a) \), \( p \in U \) and linearly extend.   

\begin{thm}\label{thm:C}
If \( f \) is of class \( B^r \), then  the linear map \( m_f: \hat{\cal B}_k^r(U)  \to \hat{\cal B}_k^r(U) \) is well-defined and continuous.   The topological vector space of	differential chains \( \hat{\cal{B_k}}^r(U)  \) is therefore a bigraded module over the ring of operators \( {\cal B}_0^r(U) \).    Furthermore, 
										  \( \|m_f A\|_{B^r} \le nr \|f\|_{B^r} \|A\|_{B^r}  \) and \( \p m_f =  m_f \p  +  m_{df} \). 
\end{thm}											
	     For a proof see  \cite{OC} Theorems 5.1.2 and 5.1.3.
  
\subsection{Mappings and the pushforward operator}  
\label{sub:operators1}  
 Suppose \( U_1 \subset \R^n\) and \( U_2 \subset \R^m \) are open  and \( F:U_1 \to
U_2 \) is a differentiable map. For \( p \in U_1 \), define \emph{linear pushforward} \( F_{p*}(v_1 \wedge \cdots \wedge
v_k) := DF_p(v_1) \wedge \cdots \wedge DF_p (v_k) \) where \( DF_p \) is the total
derivative of \( F \) at \( p \).    

 Define  \( F_*(p; \a) := (F(p), F_{p*}\a) \) for all simple $k$-elements \( (p;\a) \) and extend  to a linear map \( F_*:{\cal A}_k(U_1) \to {\cal A}_k(U_2) \) called \emph{pushforward}.     

\begin{defn}\label{def:Mr} \mbox{} \\
  Let \( {\cal M}^r(U, \R^m) \) be the vector space of differentiable maps \( F:U \to \R^m \) so that the  directional derivatives \( L_{e_j} F_i \) of its coordinate functions \( F_i \)  are of class \( B^{r-1} \), for \( r \ge 1 \). Define the seminorm  \(  |F|_{D^{r,U}} := \max_{i,j} \{\|L_{e_j} F_i\|_{B^{r-1, U}}\} \).   
	 \end{defn}       We write \(   |F|_{D^r} =   |F|_{D^{r,U}} \) when \( U \) is understood.  It is not hard to show that \( |\cdot|_{D^r} \) is a seminorm, but not a norm, on the vector space \( {\cal M}^r(U, \R^m) \).  Let \( {\cal M}^\i(U, \R^m) \) be the projective limit of \( {\cal M}^r(U, \R^m) \).     Let \( {\cal M}^r(U_1, U_2) \) denote the subset   \( \{F \in {\cal M}^r(U_1, \R^m) : F(U_1) \subset U_2 \subset \R^m\} \) where \( U_2 \) is open for all \( 0 \le r \le \i \).  If \( U \subset \R^n \) is a regular open set, we can similarly define  \( {\cal M}^r(\overline{U}, \R^m) \)  as the space of maps \( F:\overline{U} \to \R^m \) which extend to maps defined in a neighborhood of \( \overline{U} \).   

\begin{thm}\label{thm:D}
 If \( F\in {\cal M}^r(U_1, U_2) \), then   
	\[
	\|F_*(A)\|_{B^{r,U_2}} \le  mn\max \{ 1, r  |F|_{D^{r,U_1}}\} \|A\|_{B^{r,U_1}} \] for all
	\( A \in {\cal A}_0(U_1) \) and all \( r \ge 0 \). 
	and thus determines a continuous linear map \( F_*: \hat{\cal B}_k^r(U_1) \to \hat{\cal B}_k^r(U_2) \) with \( \cint_{F_* J} \o = \cint_J F^* \o \)  for all \( J \in \hat{\cal B}_k^r(U_1) \) and \( \o \in {\cal B}_k^r(U_2) \).  Furthermore, \( \p \circ F_* = F_* \circ \p \).
\end{thm}	
   For a proof see \cite{OC} Theorems 7.5.4 and 7.5.5.
	
\begin{thm}\label{thm:E}
Any affine $k$-cell \( \t \)  in \( U \) is represented by a unique differential $k$-chain \( \widetilde{\t} \in \hat{\cal B}_k(U) \)  such that  \( \cint_{\widetilde{\t}} \o = \int_\t \o \) for all \( \o \in {\cal B}_k^1(U) \).
\end{thm}   
	   For a proof see \cite{OC}, Theorem 2.11.2.

   \subsection{Cartesian wedge product}  
   \label{sub:cartesian_product}  

	Suppose \( U_1  \subset \R^n \) and \( U_2 \subset \R^m \) are open.  Let \( \iota_1:  U_1 \to U_1 \times U_2  \) and \( \iota_2: U_2 \to U_1 \times U_2   \) be the inclusions \(  \iota_1(p) = (p,0)  \) and \(  \iota_2(q) = (0,q) \).   Let \( \pi_1:U_1 \otimes U_2 \to U_1 \) and \( \pi_2:U_1 \otimes U_2 \to U_2 \)  be the projections \( \pi_i(p_1,p_2) = p_i \), \( i = 1, 2 \).
	Let \( (p;\a) \in {\cal A}_k(U_1) \) and \( (q;\b) \in {\cal A}_\ell(U_2) \). Define \( \times: {\cal A}_k(U_1)
	\times {\cal A}_\ell(U_2) \to {\cal A}_{k+\ell}(U_1 \times U_2) \) by \( \times((p; \a), (q;
	\b)) := ((p,q); \iota_{1*}\a \wedge \iota_{2_*}\b) \) where \( (p;\a) \) and \( (q; \b) \) are $k$- and
	$\ell$-elements, respectively, and extend bilinearly. We call \( P \times Q = \times(P,Q)
	\) the \emph{Cartesian wedge product\footnote{By the universal property of tensor product,  \( \times \) factors through a continuous linear map  \emph{cross product}  \(   \tilde{\times}: \hat{\cal B}_j(U_1) \otimes\hat{\cal B}_k(U_2) \to  \hat{\cal B}_{j+k}(U_1 \times U_2)\).  This is closely related to the classical definition of \emph{cross product} on simplicial chains (\cite{hatcher}, p. 278) }} of \( P \) and \( Q \). Cartesian wedge product of
	Dirac chains is associative since wedge product is associative, but it is not graded
	commutative since Cartesian product is not graded commutative.  The next result shows that Cartesian wedge product is continuous.

\begin{thm}\label{thm:G}
 Cartesian wedge product \( \times: \hat{\cal B}_k^r(U_1) \times
	\hat{\cal B}_\ell^s(U_2) \to \hat{\cal B}_{k+\ell}^{r+s}(U_1 \times U_2) \) is associative, bilinear and continuous for all open sets \( U_1 \subset \R^n, U_2 \subset \R^m \) and
	satisfies
	 \begin{enumerate} 
			\item $\|J\times K\|_{B^{r+s, U_1 \times U_2}} \le \|J\|_{B^{r,U_1}} \|K\|_{B^{s,U_2}}$ 
			\item \(  \|\widetilde{(a,b)} \times  J\|_{B^{r, \R \times U_1}} \le  |b-a|\|J\|_{B^{r,U_1}} \)  where \( \widetilde{(a,b)} \) is the $1$-chain
			representing the interval \( (a,b) \).
			\item  \( \p(J \times K) = ( \partial J) \times K + (-1)^k J
	\times ( \partial K) \). 
	\end{enumerate}
\end{thm}
        See \cite{OC} Theorem 10.1.3.
 
 \section{Poincar\'e Lemma} \label{PL}
\subsection{Chain homotopy}  
\label{sub:chain_homotopy}   Since \( \p_{k-1} \circ \p_k = 0 \),  
\begin{equation}\label{poly2}
 	\hat{\cal B}_n^s(U)   \buildrel \p_n \over \to \hat{\cal B}_{n-1}^{s+1}(U) \buildrel \p_{n-1} \over \to \cdots \buildrel \p_1 \over
	\to \hat{\cal B}_0^{s+n}(U) \buildrel \p \over \to \{0\}. 
\end{equation} is a bigraded chain complex.   Letting \( s \to \i \), the inductive limits
\begin{equation}\label{poly1}  
 	\hat{\cal B}_n(U)   \buildrel \p_n \over \to \hat{\cal B}_{n-1}(U) \buildrel \p_{n-1} \over \to \cdots \buildrel \p_1 \over
	\to \hat{\cal B}_0(U) \buildrel \p \over \to \{0\}. 
\end{equation} form a chain complex since \( \p_k \) is well-defined and continuous on the inductive limits.
 
A collection of maps \(  S_k: \hat{\cal B}_k^r(U_1) \to \hat{\cal B}_{k}^r(U_2)  \) is a \emph{graded map of chain complexes} if \( S_{k-1} \circ \p_k = \p_{k} \circ S_k  \).     

\begin{defn}\label{def:homo}  For \( U \subset \R^n \) open, the $k$-th ``differential homology'' group of the chain complex  \( (\hat{\cal B}_k(U), \p_k) \) is \[
 \hat{H}_k(U)    := \frac{\ker \p_k}{ \text{im\,} \p_{k+1}}. 
	\] 
\end{defn} 

Let \( U_\e \) be the \( \e \)-neighborhood of a smoothly embedded \( m \)-manifold \( M \) in \( \R^n \).  Define\footnote{In a sequel in preparation,  M.W. Hirsch and the author show that \( \hat{H}_k(M)   \) coincides with singular homology. They are developing a homology theory \( \hat{H}_k(J) \) for arbitrary differential $\ell$-chains \( J \in \hat{\cal B}_\ell(U) \) \( 0 \le \ell \le n \), but it is not yet clear which axioms of homology such a theory would satisfy.} \[  \hat{H}_k(M) := \varinjlim \hat{H}_k(U_\e)  \] where the inductive limit is taken as \( \e \to 0 \).
    
If \( F_k: \hat{\cal B}_k^r(U_1) \to \hat{\cal B}_k^r(U_2) \) is a graded map of chain complexes, then \( F_k \) induces a well-defined map \( \hat{H}_k(F_k): \hat{H}_k(U_1)  \to \hat{H}_k(U_2) \) with \( \hat{H}_k(F_k)[J] := \hat{H}_k[F_kJ] \) since \( F_k \) commutes with \( \p \). 

Two graded maps of complexes \( F_k, G_k:\hat{\cal B}_k^r(U_1) \to \hat{\cal B}_k^r(U_2)  \) are \emph{chain homotopic} through  a family of maps \(  K_k: \hat{\cal B}_k^r(U_1) \to \hat{\cal B}_{k+1}^r  \)  if  \( G_k - F_k =   \p K_k + K_{k-1} \p  \).  It is a standard result that if two  maps of complexes \( F_k, G_k \) are chain homotopic, then \( \hat{H}_k(F_k) = \hat{H}_k(G_k) \).

Two maps \( F_0, F_1:U_1 \subset \R^n \to U_2 \subset \R^m \) are \emph{\( B^r \) homotopic} if  there exists \( F:[0,1] \times U_1 \to U_2 \) with \( F(0,p) = F_0(p) \) and \( F(1;p) = F_1(p) \) and \( F \in {\cal M}^r([0,1] \times U_1 \to U_2) \).  We say \( U_1 \) is \( B^r \) \emph{contractible} in \( U_2 \) if \( U_1 \subset U_2 \subset \R^n \) are open and the inclusion  \( U_1 \hookrightarrow U_2 \) with \( p \mapsto p \in U_1 \) is \( B^r \) homotopic to a constant map  \( U_1 \twoheadrightarrow U_2 \) with \( p \mapsto c \in U_2 \).  

We may now state our main result. 

\begin{thm}\label{thm:poincarelemma}[Poincar\'e Lemma for Differential Chains] 
Let \( U_1 \subset U_2 \subset \R^n \) be open subsets where \( U_1 \) is \( B^r \) contractible in \( U_2 \).  Then for \( J \in \hat{\cal B}_k^r(U_1) \) with \( \p J = 0 \), there exists \( C \in \hat{\cal B}_{k+1}^r(U_2) \) with \( \p C = J \) for all \( 1 \le k \le n \) and \( 1 \le r \le \i \).  
\end{thm}     

In order to prove this we first establish   Corollary \ref{cor:ppp} which says that pushforwards through homotopic maps are homotopic as maps of complexes.  

Let \( \tilde{I} \) denote the $1$-chain representing the interval \( [0,1] \subset \R \) (see Theorem \ref{thm:E}), and \( L = L_k: \hat{\cal B}_{k}^r( U_1) \to \hat{\cal B}_{k}^r((0,1) \times U_1) \) be Cartesian wedge product \( L_k(J) := \tilde{I} \times J \) (see Theorem \ref{thm:G}).   Recall that \( (t;1) \) is the simple unit $0$-element supported in \( t \in \R^1 \).

\begin{lem}\label{lem:cart}
\( \|L J\|_{B^r} \le \|J\|_{B^r} \)  and \( (\p L + L \p)(J) = (1;1) \times J - (0;1) \times J \).
\end{lem}  

\begin{proof} 
  According to  the Test Lemma, the inequality reduces to showing
\( \| \tilde{I} \times \D_{\s^j} (p;\a))\|_{B^r} \le   \|\s\|\a\|  \) for all \( 0 \le j \le r. \)    Using Theorem \ref{thm:G} we know \( \| \tilde{I} \times \D_{\s^j} (p;\a))\|_{B^r} \le \| \D_{\s^j} (p;\a))\|_{B^r} \le \|\s\|\|\a\| \).
\end{proof}
      
Let \( F:[0,1] \times U_1 \to U_2 \) be an element of  \( {\cal M}^r([0,1] \times U_1 \to U_2) \)  with \( F(0,p) = f_0(p) \) and \( F(1;p) = f_1(p) \).   Then \( F_*:\hat{\cal B}_k^r((0,1) \times U_1) \to \hat{\cal B}_k^r(U_2) \) is a continuous linear map.  We remark that \( \widetilde{[0,1]} \times (p;1) \in \hat{\cal B}_k^r((0,1) \times U_1) \) for all \( p \in U_1 \), even though \( [0,1] \) is closed (see \cite{OC} for a full discussion about chains in open sets.).         Let \(  K = K_k :=   F_* L_k: \hat{\cal B}_{k}^r( U_1) \to \hat{\cal B}_{k+1}^{r-1}( U_2) \) for \( r \ge 0 \).

\begin{thm}\label{lem:Kone}   If \( F \in {\cal M}^r([0,1] \times U_1 \to U_2) \), then
  \( K = F_* L \) extends to a continuous linear map	\( K: \hat{\cal B}_{k}^r(U_1) \to \hat{\cal B}_{k+1}^r(U_2)  \) satisfying \( \|K J\|_{B^r} \le mn\max \{ 1, r |F|_{D^r}\} \|J\|_{B^r} \) for all differential chains \( J \in \hat{\cal B}_{k}^r(U_1) \) for all \( 0 \le k \le n-1 \).  
\end{thm} 
          
\begin{proof} 
\( \|K J\|_{B^r} = \|F_* L J\|_{B^r} \le  mn\max \{ 1, r |F|_{D^r}\}\|LJ\|_{B^r} \le mn\max \{ 1, r |F|_{D^r}\} \|J\|_{B^r} \).  
\end{proof}

\begin{thm}\label{thm:homotopy}  If  \( f_0, f_1:U_1 \subset \R^n \to U_2 \subset \R^m \) are  \( B^r \) homotopic,   the maps of chain complexes  \( f_{1*}, f_{0*}: \hat{\cal B}_k^r(U_1) \to \hat{\cal B}_k^r(U_2) \) are chain homotopic through the family of maps \( \{K_k\} \).  That is,
	\[ \p_{k+1} K_k + K_{k+1} \p_k =    f_{1*} - f_{0*}  \]  for all \( 0 \le k \le n-1 \).
\end{thm}    
\begin{proof}   By Theorems \ref{thm:D} and \ref{thm:G}
\begin{align*}
   (\p K + K \p)(J) =  (\p F_* L + F_* L \p)(J) &=     F_*(\p L + L \p)(J) 
\\&=    F_*(\p( \tilde{I} \times J) +  \tilde{I} \times\p J)  \\&
=    F_*((\p\tilde{I}) \times J - \tilde{I} \times \p J +  \tilde{I} \times\p J) \\&
= F_*((\p\tilde{I}) \times J) \\&= F_*((1;1) \times J - (0;1) \times J)  \\&= (f_{1*}   - f_{0*} )(J). 
\end{align*} 
\end{proof} 
 If \( k = 0 \), \( f_0 = c \) and \( f_1 = I \),  then \( f_{1*} - f_{0*} = I_* \), but \( \p K + K \p = \p K \ne  I \), so the theorem fails for \( k = 0 \).   
\begin{cor}\label{cor:ppp}
    Pushforwards through homotopic maps act identically on the homologies.  That is, 
\( H_k(f_{0*}) = H_k(f_{1*}) \).    In particular, on a contractible domain, \( H_k(I_*) = H_k(c_*) \), where \( c \) is the constant map.
\end{cor}

Proof of the Geometric Poincar\'e Lemma: 
\( \hat{H}_k(c_*) = 0 \) and thus \( \hat{H}_k(I_*) = 0 \).  But this implies \( \hat{H}_k(U_1) = 0 \).  That is, \( \text{im\,} \p = \ker \p \) which is what we wanted to prove.  
 
 The cone over a  differential cycle  is unique up to addition of a differential chain boundary since $J = \p (K J + \p C).$ The chain $\p C$ is   a \emph{gauge choice}. In codimension one, $L = 0.$
   \subsection{Poincar\'e Lemma for forms}\label{sub:poincare_lemma_for_forms} 
 Suppose \( F \in {\cal M}^r([0,1] \times U_1 \to U_2) \).    Let \( F_t:U_1 \to U_2 \) be given by\( F_t(p) := F(t,p) \) and \( K_k = F_* L_k \).     
 Define \( A_k \o :=  \o K_k \).
\begin{thm}\label{thm:AK}  \mbox{}   \( A_k: {\cal B}_{k+1}(U_2) \to {\cal B}_k(U_1) \)  is a continuous linear map for all \( 1 \le k \le n-1 \) and satisfies
  \begin{enumerate}
	\item \( d A_k + A_{k+1} d = F_1^* - F_0^* \);

	\item \( \|A\o\|_{B^r} \le  mn\max \{ 1, r |F|_{D^r}\} \|\o\|_{B^r} \) for all \( \o \in \hat{\cal B}_k^r(U_2) \);
	\item  \( (A \o) (p;\a) =  \int_0^1 i_{\p/\p t}    \o(F_t(p);  \a)   dt  \)                     

\end{enumerate}
\end{thm}  

\begin{proof}
(a) and (b) follow immediately from  Lemma \ref{lem:cart}  and Theorem \ref{thm:D}.

(c): 	   Recall from Theorem \ref{thm:E} that if \( \widetilde{\g} \) represents \( [0,1] \) in \( \hat{\cal B}_1^r(\R^1) \), and \( f dt \in {\cal B}_1^r(\R^1) \) is a $1$-form, then \( \cint_{\widetilde{\g}} f dt = \int_0^1 f(t)dt \).  Let \( (p;\a) \) be a simple $k$-element with \( p \in U_1 \).  Let \( \widetilde{\g_p} \) be the $k$-chain representing the interval \( [0,1] \times \{p\} \subset [0,1] \times U_1 \).   
	 If  \( \o \in  {\cal{B}}_{k+1}^r(U_2) \),  then   
	\begin{align*}
	  A \o  (p;\a) =    \cint_{K (p;\a)} \!\! \o =    \cint_{F_*L (p;\a)}\!\! \o &=  \cint_{  \widetilde{I} \times (p;\a)}\!\! F^*\o &&\mbox{ by Theorem \ref{thm:D}} \\& =  (-1)^k \cint_{  (p;\a) \times \widetilde{I} } \!\!\!\!F^*\o  &&\mbox{ by antisymmetry of } \wedge \\& 
	=  (-1)^k\cint_{  E_\a  \widetilde{\g_p}} \!\!F^*\o &&\mbox{ by the definition of } \times  \\&
	=  (-1)^k\cint_{ \widetilde{\g_p}} i_\a F^*\o   &&\mbox{ by Theorem \ref{thm:A}} \\&
	 = \cint_{\widetilde{\g_p}} i_{\p/\p t} i_\a F^* \o \wedge dt &&\mbox{ by the definition of interior product } \\&
	 = \int_0^1 i_{\p/\p t} i_\a F^* \o((t,p);1)   dt   \\&
 = \int_0^1 i_{\p/\p t}    \o(F_t(p);  \a)   dt   
	\end{align*} 
	
 \end{proof} 
\begin{remarks}  \mbox{}
	\begin{itemize}
		\item 	If \( k = 0 \)    and \( \o \in {\cal B}_1^r(U_2) \), then \(A \o    \) is a \( 0 \)-form, i.e., a function, and \(  A\o(p) =   A \o(p;  1)  = \int_0^1 i_{\p/\p t}    \o \circ F_t(p)   dt \), the classical formula for the homotopy operator of functions.  
		\item  If \( k = n -1 \) then \( \o \in {\cal B}_{n}^r(U_2) \), and thus \( \o = g d V \) where \( g \in {\cal B}_0^r(U_2) \).   It follows that  \(A \o = A g d V   \) is an \( (n-1) \)-form and \(     A (g dV)(p; \a)  = \int_0^1 i_{\p/\p t}   g dV    (F_t(p);\a) dt = \int_0^1 g(F_t(p)) dt \).  
	\end{itemize}
\end{remarks}
We immediately deduce:

 \begin{cor}
  [Poincar\'e Lemma for forms] \label{pf} Let   \( U_1 \subset U_2 \subset \R^n \) be open and \( U_1 \) be \( B^r \) contractible in \( U_2 \), and \( 2 \le k \le n-1 \).  If $\omega \in {\cal B}_k^r(U_2)$ satisfies \( d \o = 0 \), there exists \( \eta \in {\cal B}_{k-1}^r(U_1) \)  such that $\omega|_{U_1} = d \eta$.   
\end{cor}   
\begin{proof}   Let \( F: I \times U_1  \to U_2\) be the contraction.  Then \( F_{0*} = I_* \) and \( F_{1*} = 0 \). 
  Set $\eta = A\omega$ and let \( p \in U_1 \).   Applying Stokes' Theorem \ref{thm:B}  twice,   and Theorem \ref{thm:homotopy}  we have
\begin{align*}
	 d\eta(p;\a) = d(\o K) (p;\a)   =  \cint_{K \p (p;\a)}  \!\! \!\! \!\! \!\! \!\!\!\!\o   \quad =  \cint_{(p;\a)}   \!\!\!\! \!\! \o - \cint_{\p K  (p;\a)}   \!\! \!\! \!\! \!\!\!\! \!\! \o  \quad =  \o (p;\a).
\end{align*}

\end{proof}
  \section{Applications}\label{app}

\begin{lem}
  \label{cycle}  Let $n > 0$. There does not exist a   nonzero differential $n$-cycle   in a contractible open set of a smooth $n$-manifold $M$.  
\end{lem} 

\begin{proof}
  If $\partial J = 0$, then   $K\partial J = 0$.   However, $K J = 0$ since every $(n+1)$-chain in $\R^n$ is degenerate. Thus   $$J =
  \partial K J + K
  \partial J = 0.$$
\end{proof}

\subsection{Generalization of the Intermediate Value Theorem} 
\label{sub:intermediate_value_theorem_rolle_s_theorem_and_themean_value_theorem}

\begin{thm}
  \label{ivt}(General Intermediate Value Theorem) Suppose $G: U_1 \subset \R^n \to  U_2 \subset \R^m$ is an element of \( {\cal M}^r(U_1, U_2) \) where $1 \le n \le m$ and \( 1 \le r \le \i \). Let $J \in \hat{\cal B}_n^r(U_1)$   and $K \in \hat{\cal B}_n^r(U_2)$.   Then $$G_* ( \partial J) = \partial K \iff G_*J = K.$$
\end{thm}

\begin{proof}
  Suppose $G_* (\partial J) = \partial K $.  Since pushforward commutes with boundary, $\p(G_*J - K) = 0.$    By Lemma \ref{cycle} it follows that $G_*J = K.$ The converse is immediate since the boundary operator is continuous and commutes with pushforward\footnote{M.W. Hirsch helped the author clarify and simplify the original proof of the general intermediate value theorem announced as Corollary 22.8 in \cite{gmt}.}.   
\end{proof}      

This significantly strengthens the conclusion of the classical result in topology:
If  $G:B^n \subset \R^n \to \R^n$ is a continuous map whose restriction to the boundary of the ball  $B^n$ is the identity map, than the image of $g$ contains  all of $B^n$.  Our result shows that if $G$ is Lipschitz, then $G_*\widetilde{B^n} = \widetilde{B^n}$.  For $n=m=1$, this generalizes the intermediate value theorem (see Figure    \ref{fig:funky}). 
                                   
      \begin{figure}[htbp]
      	\centering
      		\includegraphics[height=3in]{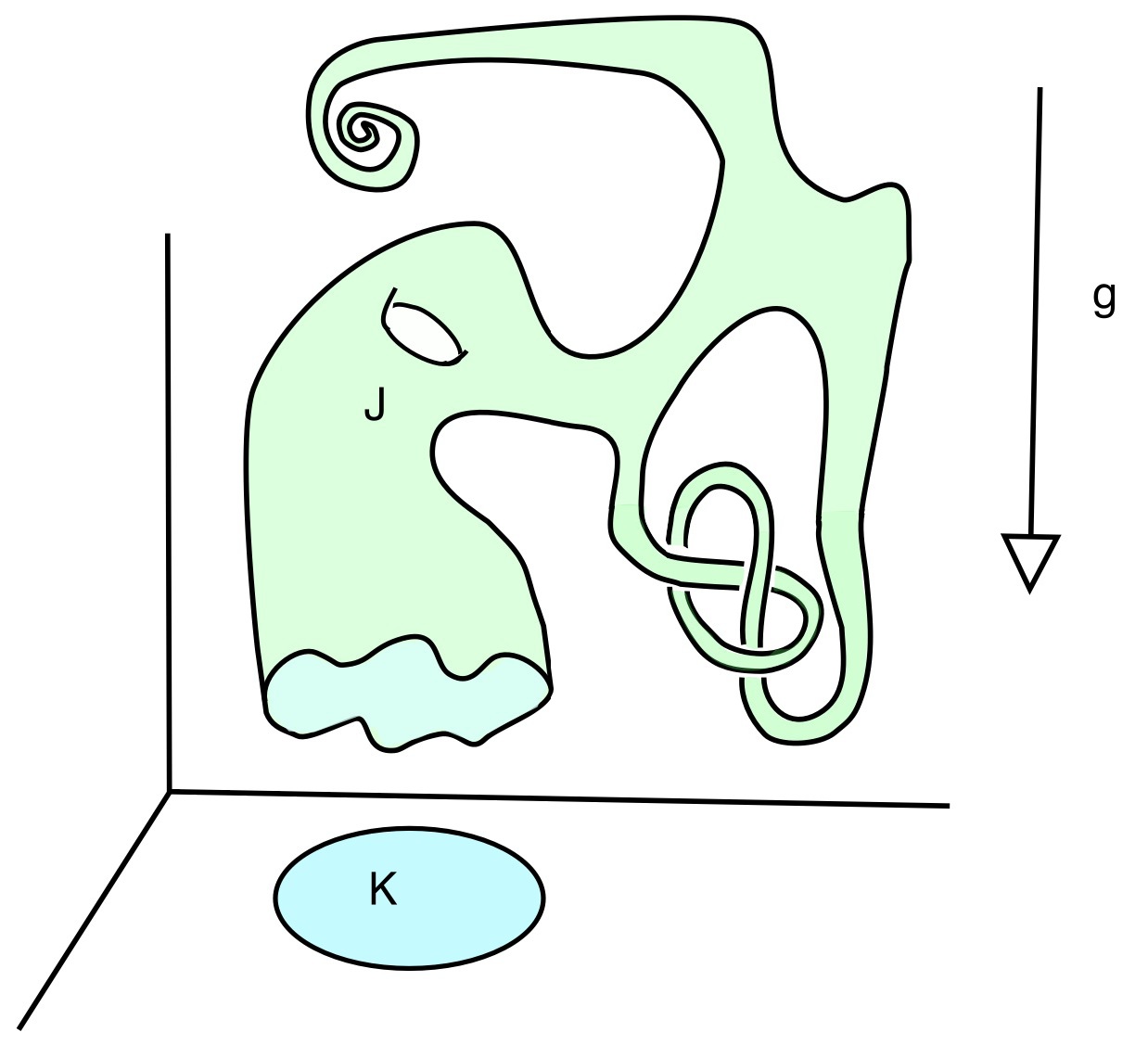}
      	\caption{General intermediate value theorem}
      	\label{fig:funky}
      \end{figure}
      
\begin{cor}\label{thm:jordan}
 Suppose \( \p J = 0 \) and \( J  \in \hat{\cal B}_k(U) \) is supported in a contractible open set \( U \).  Then \( K J = L  \) if and only if  \( J = \p L \).
\end{cor} 

\begin{proof}
	Suppose \( KJ = L \).  Then \( \p KJ = \p L \).  But \( J = (K \p + \p K)J = \p K J = \p L \). Conversely, suppose \( J = \p L \).  Then \( (K \p + \p K) J = \p KJ = \p L  \). By Theorem \ref{ivt} \( KJ = L \). 
\end{proof} 

\section{Discrete Poincar\'e Lemma}  
\label{sec:discrete_poincar'e_lemma}
 Discrete versions of the Poincar\'e Lemmas  are readily available since Dirac chains are dense in the space of differential chains.   If \( J \in \hat{\cal B}_k(U) \), we can approximate \( J  \) with \( A = \sum (p_i;\a_i) \) and apply the operator \( K \) to each $k$-element \( (p_i;\a_i) \).  Since \( K \) is linear and continuous,     \( K J \) is approximated by \( \sum K(p_i;\a_i) \).  The differential complex \( {\cal A}_k^j(U) \) of Dirac chains of arbitrary order and dimension in \( U \) (see \cite{OC} \S 3.4) is the discrete analogue of \( \hat{\cal B}_k(U) \).    Since the other operators we use, e.g.,  boundary, are also linear and continuous,  we obtain discrete versions of all of the results of this paper.     If we fix a finite set of ``base points'' \( \{p_i\} \), the space of Dirac chains becomes finite dimensional and the operators can be represented as matrices.

\providecommand{\bysame}{\leavevmode\hbox to3em{\hrulefill}\thinspace}
\providecommand{\MR}{\relax\ifhmode\unskip\space\fi MR }
\providecommand{\MRhref}[2]{%
  \href{http://www.ams.org/mathscinet-getitem?mr=#1}{#2}
}
\providecommand{\href}[2]{#2}   

\addcontentsline{toc}{section}{References} 
\bibliography{Jennybib.bib, mybib.bib}{}
\bibliographystyle{amsalpha}

\bibliographystyle{amsalpha}

\end{document}